\documentclass[arxiv,locbib]{myarticle}

\initbibli

\newcommand{\transl}[2]{{^{#1}#2}}
\newcommand{\Lstpop}{\mathrm{Lstp}}
\newcommand{\Lstp}{\Lstpop}
\newcommand{\stabop}{\mathrm{Stab}}
\NewSymbol[function=true]{stab}{\stabop}{stabilizer}

\newcommand{\dcltL}{\dcl[\tL]}
\NewSymbol{LGD}{\LGN[\D]}{Language of $\VDFG$}
\newcommand{\dclLL}{\dcl[\LL]}
\NewSymbol[exponant={\Geom}]{SCVHG}{\mathrm{SCVH}}{Separably closed valued fields with Hasse derivations in the geometric language}
\NewSymbol[index={\mathrm{P}}]{LP}{\LL}{Language of pairs}
\NewSymbol{Balls}{\mathbb{B}}{balls}

\title{A short note on groups in separably closed valued fields}

\author{Silvain Rideau\thanks{Partially supported by ValCoMo (ANR-13-BS01-0006)}}
\date{\today}

\begin{document}

\maketitle

\begin{abstract}
In this note we show that groups with definable generics in a separably closed valued of finite imperfection degree can be embedded into groups definable in their algebraic closure.
\end{abstract}

In \cite{Hru-UniDim}, Hrushovski showed that a pro-definable group (i.e. a pro-definable set with a pro-definable groups law) in a stable theory is isomorphic to a pro-limit of definable groups. In other terms, the two natural notions of "groups in infinitely many variables" --- (pro-definable) groups, i.e. group objects in the category of definable groups, and pro-(definable group), an object in the pro-category of definable groups --- coincide in structure whose theory is stable. In \cite{HruRid-Meta}, it is shown that this result extends, in any theory, to any pro-definable group with a \(d\)-generic, that is a definable type concentrating on the group with boundedly many translates action of \(G\) (cf. \ref{dgen}).

The second result of \cite{Hru-UniDim} which is generalized in \cite{HruRid-Meta}, is the possibility of reconstructing groups out of generic data. The idea is initially due to Weil \cite{Wei-GpCh} in the setting of algebraic groups. It was then transposed to a more general topological setting in \cite{vdD-GpCh} and to the stable setting in \cite{Hru-UniDim}. Here, we will be considering the slightly more general notion of pro-definable group chunks which first appeared in \cite{Pil-DCFGp}. Put together, these two results show that for any theory \(T\), the inclusions between the categories of pro-limits of definable groups with \(d\)-generics, pro-definable groups with \(d\)-generics and group chunks over definable types, are equivalences. Note that for this equivalence to hold of all groups and not just the "connected" ones, we have to consider definable partial types.

Our goal in this note is to use those results to study groups definable in separably closed valued fields. In \cite{Pil-DCFGp}, Pillay showed, using the reconstruction of groups from generic data, as well as the fact that pro-definable groups in algebraically closed fields are pro-limits of definable groups, that every group definable in a differentially closed field of characteristic zero can be definably embedded in a group definable in the underlying algebraically closed field (i.e. an algebraic group). A similar result was then proved in \cite{BouDel-SCFGp} for the groups definable in separably closed fields of finite imperfection degree. In both these cases, both theories involved are stable and Hrushovski's tools from \cite{Hru-UniDim} can be used.

In this note we give an abstract version of these proofs, \ref{embed group}, by showing that, under certain hypothesis, if the definable closure in a theory \(T\) is bounded by the definable closure in a theory \(T_0\), then groups with definable generic in \(T\) embeds in groups definable in \(T_0\). This is, to the best of the author's knowledge, the only existing such embedding result that does not require \(T_0\) to be stable. This result is then applied, in \ref{gp SCVH}, to prove that all groups interpretable in a separably closed valued field of finite imperfection degree, with a \(d\)-generic, can be definably embedded in a group interpretable in the algebraic closure (as a valued field).

The paper is organized as follows. In Section\,\ref{sec:pro-def}, we discuss pro-definable sets. In Section\,\ref{sec:generics}, we define the notion of \(d\)-generics and describe some of their properties. In Section\,\ref{sec:gp ch}, we explain how to reconstruct a group out of generic data over a definable type. In Section\,\ref{sec:gp enrich} we give an abstract criterion for groups with definable generics to be definably embeddable in a group definable in a reduct. In Section\,\ref{sec:SCVH}, we apply those results to separably closed valued fields of finite imperfection degree.

\section{Pro-definable sets}
\label{sec:pro-def}

In what follows we will be mostly considering "definable sets in infinitely many variables". There are two equivalent formalisms: partial types in infinitely many variables or formal filtered projective limits of definable sets. We will, in this note, prefer the second formalism. Recall that an order is filtered if any two points have an upper bound. A pro-definable set is a family \((X_i)_{i\in I}\) of definable sets with \(I\) a filtered order, and for any \(i < j \in I\), a definable map \(f_{j,i}: X_j \to X_i\). We consider this family as the formal projective limit \(\projlim X_i\). If we want to specify the language, we will say that \(X\) is pro-\(\LL\)-definable meaning that all the \(X_i\) and the transition maps \(f_{j,i}\) are \(\LL\)-definable. A pro-definable map \(f :\projlim X_i \to \projlim Y_j\) is a family of definable maps \(f_j : X_{i_j} \to Y_j\) which is compatible with the projective system. In other words, it is an element of \(\injlim_i \projlim_j\Hom{X_i}{Y_j}\).

Note that if \(X = \projlim X_i\) is pro-definable, \(x_i\) is a tuple of variables sorted like \(X_i\) and \(x = (x_i)_{i\in I}\), then we can consider \(X\) as a partial type in the variable \(x\) (which states that \(x_i = X_i\) and \(f_{j,i}(x_j) = x_i\)). In particular, we can consider (partial) types \(p(x)\) which concentrate on \(X\). Following usual model theoretic terminology, when all the maps \(f_{j,i}\) of a given projective system \(\projlim X_i\) are injective then we say that \(\projlim X_i\) is \(\infty\)-definable. Note that in that case, we can and will consider that all the \(X_i\) are subsets of a given \(X_{i_0}\). It is then natural to consider that \(\projlim X_i\) is nothing else than the intersection of the \(X_i\).

Note that, by compactness, a pro-definable map between pro-definable sets is exactly a function whose graph is pro-definable. Also, a pro-definable map between \(\infty\)-definable sets is the restriction to an \(\infty\)-definable set of a definable map.

\section{Definable generics}
\label{sec:generics}

Let $T$ be any $\LL$-theory that eliminates imaginaries and $M\models T$ be sufficiently saturated and homogeneous and $A\subseteq M$. Let $(G,\cdot)$ be a pro-\(\LL(A)\)-definable group, i.e. a pro-\(\LL(A)\)-definable set with a pro-\(\LL(A)\)-definable group law.

\begin{definition}(definable partial type)
Let \(p(x)\) be a partial type over \(M\) (in the possibly infinite tuple of variables $x$). We say that $p$ is $\LL(A)$-definable if for every formula $\phi(x;y)$ there is an $\LL(A)$-formula $\theta(y)$, usually denoted $\defsc{p}{x}\,\phi(x;y)$, such that for all tuple $m\in M$:
\[p(x) \vdash \phi(x;m) \text{ if and only if } M\models\defsc{p}{x}\,\phi(x;m).\]
\end{definition}

As the definition above makes explicit, the partial types that appear in this note are considered to be closed under implication. In particular, if $p(x)$ is a partial type and $\phi(x;y)$ an $\LL$-formula, $\restr{p}{\phi}$ denotes $\{\phi(x;m) \mid p(x)\vdash \phi(x;m)\}$ and $\restr{p}{A}$ denotes $\{\phi(x)\in\LL(A) \mid p(x)\vdash \phi(x;m)\}$.

\begin{definition}($\transl{g}{p}$)
Let $p(x)$ be a partial type over $M$ concentrating on \(G\) and $g\in G(M)$. We define \[\transl{g}{p} := \{\phi(x;y)\mid p\vdash\phi(g\cdot x;y)\}.\]
\end{definition}

\begin{remark}
\begin{thm@enum}
\item The realizations of $\transl{g}{p}$ are exactly the elements of the form $g\cdot x$ for some $x\models p$.
\item If $p$ is $\LL(A)$-definable, then \(\transl{g}{p}\) is $\LL(Ag)$-definable and we can choose \(\defsc{\transl{g}{p}}{x}\,\phi(x;y)\) to be \(\defsc{p}{x}\,\phi(g\cdot x;y).\)
\item This operation defines an action of $G(M)$ on partial types over $M$ which concentrate on $G$.
\end{thm@enum}
\end{remark}

Following \cite{HruRid-Meta}, we want to consider groups with a definable generic (recall that $A\subseteq M$ is supposed to be small):

\begin{definition}[dgen]($d$-generic type)
Let $p(x)$ be a partial type over $M$ concentrating on $G$. We say that $p$ is a $d$-generic of $G$ over $A$ if for all $g\in G(M)$, $\transl{g}{p}$ is $\LL(A)$-definable.
\end{definition}

When we do not want to specify the (small) set of parameters $A$, we will simply say that $p$ is a $d$-generic in $G$.

\begin{remark}
The notion of genericity that is usually considered  in unstable contexts (see \cite{NewPet,HruPil-Nip}) is defined using forking: a $p(x)$ partial type over $M$ concentrating on $G$ is said to be $f$-generic over $A$ if for all $g\in G(M)$, $\transl{g}{p}$ does not fork over $A$. If $T$ is $\NIP$ and $A  = \acl(A)$, a complete $d$-generic type is nothing more than a definable $f$-generic type. Indeed, in an $\NIP$ theory, a complete type which is non forking over $A$ is \(\Lstp(A)\)-invariant and hence, if it is definable, its definable scheme is over $\acl(A) = A$.
\end{remark}

The main property of pro-definable groups with \(d\)-generics that we will be using is that pro-definable groups with $d$-generics are pro-limits of definable groups.

\begin{proposition}[star def gp prolim]({\cite[Proposition\ 3.4]{HruRid-Meta}})
Let $G$ be a pro-\(\LL(A)\)-definable group. Assume that $G$ admits a partial type $p$ $d$-generic over $A$. Then there exists a projective system of $\LL(A)$-definable groups $(H_{\alpha})_{\alpha\in A}$ and a pro-$\LL(A)$-definable group isomorphism $f : G\to H := \projlim_{\alpha} H_{\alpha}$.
\end{proposition}

\begin{remark}
There is a classic counter-example to \ref{star def gp prolim} when the group does not have a $d$-generic. Let $M$ be an $\aleph_{0}$-saturated real closed field. The group $\mathcal{I}$ of infinitesimal elements $\{x\in M\mid\forall n\in\Zz_{>0},\,-\frac{1}{n} < x < \frac{1}{n}\}$ is an $\infty$-definable subgroup of the additive group $\Gg_{a}(M)$ but there is no proper definable subgroup of $\Gg_{a}$ containing $\mathcal{I}$. 
\end{remark}

\section{Group chunks}
\label{sec:gp ch}

Let us now consider group chunks, a fundamental tool to construct groups in model theory. Recall that if $p(x)$ is a partial type over $M$ and $f$ is a pro-\(\LL(M)\)-definable function defined on \(p(x)\), \[\push{f}{p}(y) := \{\phi(y;z)\mid \phi(f(x);z)\in p\}.\] If $p$ is $\LL(A)$-definable and $f$ is pro-$\LL(A)$-definable, then $\push{f}{p}$ is $\LL(A)$-definable and we can choose $\defsc{\push{f}{p}}{y}\,\phi(y;z) = \defsc{p}{x}\,\phi(f(x);z)$. Also if $p(x)$ and $q(y)$ are partial $\LL(A)$-definable types, let \[p\tensor q = \{\phi(x,y;m)\mid \forall a\models \restr{p}{Am}\,\forall b\models \restr{q}{A m a},\,M\models\phi(a,b;m)\}.\]
Note that $p\tensor q$ is also $\LL(A)$-definable and that we can choose:
\[\defsc{p\tensor q}{x y}\,\phi(x,y;s) := \defsc{p}{x}\,(\defsc{q}{y}\,\phi(x,y;s)).\]

Note that we are not considering group chunks on complete types, but on partial types (as does Wagner in \cite[Theorem 4.7.1]{Wag-Simple}). Restricting oneself to complete types only allows generic reconstruction of connected groups. Considering partial types allows the generic reconstruction of non-connected groups.

\begin{definition}[gp chunk](Group chunk)
Let $x$ be a possibly infinite tuple and $p(x)$ an $\LL(A)$-definable partial type. A pro-$\LL(A)$-definable group chunk over $p$ is a triple $(F,H,K)$ of pro-$\LL(A)$-definable maps defined on $p^{\tensor 2}$ such that:
\begin{thm@enum}
\item\label{gp ch:inv} For all \(a\models\tprestr{\pi}{A}\), \(\push{(F_a)}{p} = p\), where \(F_a(x) = F(a,x)\);
\item\label{gp ch:inj} $p^{\tensor 2}(x,y) \vdash H(x,F(x,y)) = y$ and $p^{\tensor 2}(x,y) \vdash K(F(x,y),y) = x$;
\item\label{gp ch:assoc} $p^{\tensor 3}(x,y,z) \vdash F(x,F(y,z)) = F(F(x,y),z)$.
\end{thm@enum} 
\end{definition}

\begin{remark}
The data describing a group chunk is somewhat redundant. Condition\,\ref{gp ch:inj} could be replaced by:
\begin{itemize}
\item[(ii)'] For all $(a,b)\models p^{\tensor 2}$, $b\in\dcl(A,a,F(a,b))$ and $a\in\dcl(A,F(a,b),b)$.
\end{itemize}
and not mention $H$ and $K$.
\end{remark}

\begin{example}
Let $G$ be a pro-definable group with a $d$-generic $p$, by \cite[Remark\ 3.3]{HruRid-Meta}, we may assume that $p$ is $G(M)$-invariant. Then the group law induces a pro-definable group chunk on $p$.
\end{example}

The converse is also true:

\begin{proposition}[gp from chunk]({\cite[Proposition\ 3.15]{HruRid-Meta}})
Let $p(x)$ be a partial $\LL(A)$-definable type and $(F,H,K)$ be a pro-$\LL(A)$-definable group chunk over $p$. Then there exists a pro-$\LL(A)$-definable group $(G,\cdot)$ and a pro-$\LL(A)$-definable one-to-one function $f$ such that $\push{f}{p}$ is an $\LL(A)$-definable $G(M)$-invariant type and $p^{\tensor 2}(x,y) \vdash f(F(x,y)) = f(x)\cdot f(y)$.
\end{proposition}

Furthermore, there is an equivalence of categories between groups with $d$-generics and pro-definable group chunks over definable types:

\begin{proposition}[gp mor from chunk]({\cite[Proposition\ 3.16]{HruRid-Meta}})
Let $(G,\cdot)$ and $(H,\cdot)$ be two pro-$\LL(A)$-definable groups, $p$ be a $G(M)$-invariant partial type over $M$ and $f_{0}$ be a pro-$\LL(A)$-definable function. If $\push{(f_0)}{p(x)} \vdash y \in H$ and $p^{\tensor 2}(x,y)\vdash f_0(x\cdot y) = f_0(x)\cdot f_0(y)$, then there exists a unique pro-$\LL(A)$-definable group morphism $f : G \to H$ such that for all $p\in P$, $p(x)\vdash f(x) = f_{0}(x)$.

Moreover, if $f_{0}$ is one-to-one, so is $f$.
\end{proposition}

Injectivity is not proved in \cite{HruRid-Meta} but is easy to check (and follows, in fact from the equivalence of categories).

\section{Groups in enrichments}
\label{sec:gp enrich}

Let us now use all of our tools to show that certain groups with a $d$-generic are in fact definable in a reduct, provided the group law is not too far from being definable in the reduct.

\begin{proposition}[embed group]
Let $\LL\subseteq \tL$ be two languages, $\Real$ be the set of $\LL$-sorts, $T$ be an $\LL$-theory which eliminates quantifiers and imaginaries and $\tT\supseteq T_\forall$ be an $\tL$-theory. Let \(\tM\models\tT\) be sufficiently saturated and homogeneous. Let \(M\models T\) containing \(\tM\) and such that any automorphism of \(\tM\) extends to an automorphism of \(M\). Let \(\tA\subseteq\tM\) be such that $\Real(\dcltL(\tA)) = \Real(\tA)=: A$. Let $(G,\cdot)$ be an $\tL(\tA)$-definable group. Assume:
\begin{thm@enum}
\item\label{hyp:gen} The group \(G\) has a $d$-generic type $p\in\TP<\tL>(\tM)$ over $\acltLeq(\tA)$;
\item\label{hyp:tame dcl} There exists a pro-$\tL(\tA)$-definable one-to-one function $f$ and pro-$\LL(A)$-definable functions $m$ and $i$ such that for all $g_{1}$, $g_{2}\in G$, $f(g_{1}\cdot g_{2}) = m(f(g_{1}),f(g_{2}))$ and $f(g_{1}^{-1}) = i(f(g_{1}))$;
\item\label{hyp:ext def} For any \(\tL(\tM)\)-definable $p\in\TP<\tL>(\tM)$, there exists $q_p\in\TP<\LL>(M)$ which is \(\LL(\tM)\)-definable and such that \(\restr{p}{\LL} = \restr{q}{\tM}\). Moreover if \(f\) is a (pro)-\(\LL(\tM)\)-definable function defined on \(p\), \(\tsig\in\aut[\tL](\tM)\) and \(\sigma\in\aut[\LL](M)\) extends it, then \(\sigma(\push{f}{q_p}) = q_{\sigma(\push{f}{p})}\).
\item\label{hyp:descent} For all \(e\in\dclLL(\tM)\), there exists \(c\in \tM\) such that \(\dclLL(e) = \dclLL(c)\).
\end{thm@enum}

Then there exists an $\LL(A)$-definable group $H$ (in \(\tM\)) and an $\tL(\tA)$-definable one-to-one group morphism $h : G(\tM)\to H(\tM)$.
\end{proposition}

\begin{proof}
Let $A := \Real(\tA)$, $P := \{\tsig(\transl{g}{p})\mid g\in G(\tM)\text{ and }\tsig\in\aut[\tL](\tM)[\tA]\} = \{\transl{g}{(\sigma(p))}\mid g\in G(\tM)\text{ and }\tsig\in\aut[\tL](\tM)[\tA]\}$, \(Q := \{q_{\push{f}{r}} \mid r\in P\}\) and $q = \bigcap_{s\in Q} s$. Since every $s\in Q$ is $\LL(\tM)$-definable and, by \cite[Proposition\ 3.2]{HruRid-Meta}, there are only finitely many \(\phi\)-types involved for any formula \(\phi\), \(q\) is also $\LL(\tM)$-definable. Moreover, for all \(\tsig\in\aut[\tL](\tM)[\tA]\), all \(\sigma\in\aut[\LL](M)\) extending it and all \(s\in Q\), \(\sigma(s)\in Q\). It follows that \(\sigma(q) = q\). As the canonical basis of \(q\) can be assumed to be in \(\tM\) by Hypothesis\,\ref{hyp:descent}, \(q\) is \(\LL(A)\)-definable.

Let $m_{1}(x,y) := m(i(x),y)$ and $m_{2}(x,y) := m(x,i(y))$.

\begin{claim}
The tuple $(m,m_{1},m_{2})$ is a pro-$\LL(A)$-definable group chunk over $q$.
\end{claim}

\begin{proof}
For all \(g\in G(\tM)\), \(q_{\push{f}{(\transl{g}{p})}} = \push{(m_{f(g)})}{q_p}\), and thus \(\push{(m_{f(g)})}{q} = q\), i.e. for every \(\LL\)-formula \(\phi(x;t)\), \(\models\forall t\,(\defsc{q}{x}\,\phi(m(f(g),x);t)\iffform\defsc{q}{x}\,\phi(x;t)) := \theta(g)\). Since \(\theta\) is an \(\LL(A)\)-formula, it is in \(q\). This is exactly Condition\,\ref{gp ch:inv}.

For all $x$, $y$ and $z\in G$, we have:
\[m(f(x),m(f(y),f(z))) = f(x\cdot y\cdot z) = m(m(f(x),f(y)),f(z)),\]
\[\begin{array}{rcl}
m_{1}(f(x),m(f(x),f(y))) &=& m(i(f(x)),m(f(x),f(y)))\\
&=& m(f(x^{-1}),f(x\cdot y))\\
&=& f(x^{-1}\cdot x\cdot y)\\
&=& f(y).
\end{array}\]
Similarly, $m_{2}(m(f(x),f(y)),f(y)) = f(x)$. It follows that Conditions\,\ref{gp ch:inj} and \ref{gp ch:assoc} also hold.
\end{proof}

By \ref{gp from chunk}, there exists a pro-$\LL(A)$-definable group
$(L,\cdot)$ and a pro-$\LL(A)$-definable one-to-one function $l$ such
that $q^{\tensor 2}(x,y) \vdash l(m(x,y)) = l(x)\cdot l(y)$ and $\push{l}{q}$ is an $\LL(A)$-definable $L(N)$-invariant type. By \ref{star
  def gp prolim}, there exists a projective system of
$\LL(A)$-definable groups $(H_{\beta},\cdot)_{\beta\in B}$ and a
pro-$\LL(A)$-definable groups isomorphism $j$ between $L$ and
$\projlim H_{\beta}$. For all $\beta\in B$, let $\pi_{\beta} :
\projlim H_{\beta} \to H_{\beta}$ be the canonical projection and
$h_{\beta} = \pi_{\beta}\comp j\comp l\comp f$. Since $\projlim
h_{\beta}$ is one-to-one, by compactness (in \(\tM\)), there exists
$\beta_0 \in B$ such that $h_{\beta_0}$ is already one-to-one.

Let \(u := \bigcap_{r\in P}r\). Note that \(\push{f}{u} \vdash \restr{q}{M}\). By Hypothesis\,\ref{hyp:descent}, we can find an \(\LL(A)\)-definable bijection \(k:H_{\beta_0} \to H'\) such that for any \(a\models u\), \(k(h_{\beta_0}(a))\in\tM\). Moreover, \(H'\) can be made into an \(\LL(A)\)-definable group such that \(k\) is a group isomorphism. Let \(h' = k \comp h_{\beta_0}\). Then $h'$ is $\tL(\tA)$-definable and, $t^{\tensor 2}(x,y)\vdash h'(x \cdot y) = h'(x)\cdot h'(y)$. The partial type \(u\) is \(G(\tM)\)-invariant and \(\LL(\tA)\)-definable type. Therefore, by \ref{gp mor from chunk}, there exists an $\tL(\tA)$-definable one-to-one group morphism $G(\tM) \to H'(\tM)$.
\end{proof}

\begin{remark}
\begin{thm@enum}
\item Note that the above proposition requires the existence of a complete $d$-generic type, but the proof actually makes use of partial $d$-generics to be able to work over bases that are not models or even algebraically closed. Note also that, in a definable group, the existence of a partial $d$-generic is an empty assumption since the group itself is a partial $d$-generic.
\item Instead of assuming that $m$ and $i$ exist, by compactness, it suffices to know that:
\begin{itemize}
\item[(ii')] There exists a pro-$\tL(\tA)$-definable one-to-one function $f$ such that for all $g_{1}$ and $g_{2}\in G$, $f(g_{1}\cdot g_{2})\in\dclLL(A,f(g_{1}),f(g_{2}))$ and $f(g_{1}^{-1})\in \dcl[\LL](A,f(g_{1}))$.
\end{itemize}
\item Note that our setting is slightly more complicated than often considered since \(\tT\) only contains \(T_\forall\) and not \(T\) itself. If we assume that \(T\subseteq\tT\), then Hypothesis\,\ref{hyp:descent} is trivial. Moreover, taking \(\tM = M\), we may replace Hypothesis\,\ref{hyp:ext def} by:
\begin{itemize}
\item[(iii')] For any \(\tL(\tM)\)-definable $p\in\TP<\tL>(\tM)$, \(\restr{p}{\LL}\) is \(\LL(\tM)\)-definable.
\end{itemize}
\end{thm@enum}
\end{remark}

As a first corollary, let us reprove the aforementioned result about groups definable in differentially closed fields of characteristic zero. Recall that $\Lrg$ is the language of rings and $\LDrg := \Lrg\cup\{\D\}$ is the language of differential rings.

\begin{corollary}[DCF]
Let $K\models\DCF[0]$, $k \substr K$ a differential field and $G$ an $\LDrg(k)$-definable group, then $G$ embeds $\LDrg(k)$-definably into an $\Lrg(k)$-definable group.
\end{corollary}

\begin{proof}
Note that for all $C\subseteq K\models\DCF[0]$, $\dcl[\LDrg](C) = \dcl[\Lrg](\prol[\omega](C))$. In particular, we have that $\dcl[\LDrg](k) = k$ and the multiplication and inverse in $G$ are of the right form to apply \ref{embed group}. As $\DCF[0]$ and $\ACF[0]$ are \(\omega\)-stable, any type of maximal Morley Rank in $G$ is $d$-generic over $\acl[\LDrg](k)$ and, since \(\ACF[0]\) is stable, the reduct of a \(\DCF[0]\)-type is obviously definable. Applying \ref{embed group}, we find an $\LDrg(k)$-definable embedding of $G$ into an $\Lrg(k)$-definable group $H$.
\end{proof}

\begin{remark}
\begin{thm@enum}
\item It follows from \cite[Théorème\,3.7.(iii) and Corollaire 3.13.(i)]{SGA-Art-Weil} that every $\Lrg(k)$-definable group chunk is $\Lrg(k)$-definably isomorphic to the group chunk of an algebraic group over $k$. In particular the group $G$ in \ref{DCF} is $\LDrg(k)$-definably embedded in an algebraic group over $k$.
\item The result obtained here is slightly more general than the one in \cite{Pil-DCFGp}. Indeed, Pillay only proves it for connected groups. The reduction to connected groups requires to work over a model of \(\DCF_0\) to be able to pick points in every coset of the connected component and thus recover the whole group. The above proof circumvents this issue by working directly with partial types and non-connected groups and allows to obtain the algebraic group and the isomorphism over any differential field.
\end{thm@enum}
\end{remark}

\section{Separably closed valued fields}
\label{sec:SCVH}

\ref{embed group} can also be applied to separably closed fields of finite imperfection degree, but we prefer to give a similar example in which the generalization of the above results to the unstable context is necessary: separably closed valued fields. We will be needing a few preliminary results.

The geometric language \(\LG\) for valued fields consists of a sort \(\K\) interpreted as the field itself, sorts \(\Latt[n]\) interpreted as \(\GL{n}(\K)/\GL{n}(\Val)\) (where \(\Val\) denotes the valuation ring) and \(\Tor[n]\) interpreted as \(\GL{n}(\K)/\GL{n,n}(\Val)\) (where \(\GL{n,n}(\Val)\) consists of those matrices in \(\GL{n}(\Val)\) whose last column modulo the maximal ideal \(\Mid\) consist of only \(0\) except for a \(1\) on the diagonal). The geometric language has the field language on \(\K\) and the canonical projections \(\sigma_n:\K<n^2> \to \Latt[n]\) and \(\tau_n : \K<n^2>\to \Tor[n]\). Actually for technical reasons, we want the \(\LG\)-theory \(\ACVFG\) of algebraically closed valued fields to eliminate quantifiers so the \(\LG\) we consider here is the Morleyization of the language we just described. For a more throughout description of the geometric language, the reader can refer to \cite{HasHruMac-ACVF}.

The main reason one introduces the geometric sorts is the following theorem:

\begin{theorem}(\cite[Theorem 1.0.1]{HasHruMac-ACVF})
The \(\LG\)-theory \(\ACVFG\) of algebraically closed valued fields eliminates imaginaries.
\end{theorem}

In \cite{HilKamRid}, it is shown that this results extends to the separably closed setting. Since this the most general result, we will be working with commuting Hasse derivations.

\begin{definition}(Iterative Hasse derivations)
An iterative Hasse derivation on a field \(K\) is a sequence of linear operators \((D_i)_{i\in\Zz_{\geq 0}}\) such that, for all \(i\), \(j\in\Zz_{\geq 0}\):
\begin{thm@enum}
\item \(D_0(x) = x\);
\item \(D_i(x y) = \sum_{j+k = i} D_j(x)D_k(y)\);
\item \(D_{i+j}(x) = \binom{i+j}{i}D_i(D_j(x))\).
\end{thm@enum}

Two Hasse derivations \(D_1 = (D_{1,i})_{i\geq 0}\) and \(D_2 = (D_{2,i})_{i\geq 0}\) are said to commute if for all \(i\), \(j\geq 0\), \(D_{1,i}\comp D_{2,j} = D_{2,j}\comp D_{1,i}\).
\end{definition}

Let \(p\) be a prime and \(e\) be non-negative, the theory \(\SCVHG[p,e]\) denotes the theory of separably closed valued fields with \(e\) commuting iterative Hasse derivations, such that \([\K:\K<p>] = p^e\) and \(\K<p> = \{x\in \K\mid \forall n \leq e,\,D_{n,1}(x) = 0\}\), in the language \(\LGD := \LG \cup\{D_{n,i}\mid 0 < n \leq e, i\geq 0\}\).

\begin{theorem}[fact SCVH](\cite[Corollary\,4.20 and Proposition\,5.5]{HilKamRid})
\begin{thm@enum}
\item\label{EI SCVH} The theory \(\SCVHG[p,e]\) eliminates imaginaries;
\item\label{dcl SCVH} For all \(A\substr M\models \SCVHG[p,e]\), \(\dcl[\SCVHG[p,e]](A) = \dcl[\ACVF](A)\).
\end{thm@enum}
\end{theorem}

Let us now show how to extend definable types over \(M\models\SCVHG[p,e]\) to definable types over \(\alg{M}\).

\begin{lemma}[ext def]
Let \(M\models\SCVHG[p,e]\) and \(p\in\TP<\LGD>(M)\) be \(\LGD(M)\)-definable. Let \(\LP := \LGD\cup\{P\}\). We make \(\alg{M}\) into an \(\LP\)-structure by interpreting \(\LG\) in \(\alg{M}\), the predicate \(P\) as \(M\) and the functions \(D_{i,n}\) in \(M\). Then there exits \(q\in\TP<\LG>(\alg{M})\) which is \(\LP(M)\)-definable and such that \(p\vdash q\).
\end{lemma}

\begin{proof}
Let \(a\models p\) and \(a' := \acl[\SCVHG[p,e]](a)\). Note that \(\tp[\LGD](a'/M)\) is also \(\LGD(M)\)-definable. Moreover, there exists a definable map \(f\) whose domain \(\K<n>\) such that \(a\in f(\K<n>)\). By density of definable types in \(\SCVHG[p,e]\) (cf. \cite[Theorem\,4.19]{HilKamRid}), there exists an \(\LGD(a')\)-definable \(\LGD\)-type \(r\) such that \(r(x)\vdash f(x) = a\). We have that \(\restr{r}{M}\) is \(\LGD(M)\)-definable and if we find \(s\in\TP<\LG>(\alg{M})\) which is \(\LP(M)\)-definable and such that \(\restr{r}{M}\vdash s\). then \(q := \push{f}{s}\) has the required property.

So we may assume that \(p\) concentrates on \(\K<n>\) for some \(n\). We proceed by induction on \(n\). Let \((a,c)\models p\) where \(\card{c} = 1\). By induction, there exists \(q\in\TP<\LG>(\alg{M})\) which is \(\LP(M)\)-definable and such that \(\tp[\LGD](a/M)\vdash q\). Let \(\Balls\) be the set of balls (open and closed, with radius in \(\Gamma\cup\{-\infty,+\infty\}\)), \(\Balls<[l]>\) the set of subsets of \(\Balls\) of cardinality at most \(l\) and \(\Balls[n]<[l]>(M)\) the set of \(\LG(M)\)-definable maps \(f : \K<n>\to\Balls<[l]>\).

Let \(E := \{f \in \Balls[n-1]<[l]>(M)\mid c\in \bigcup_{b\in f(a)}b\}\). Then the generic type of \(E\) over \(q\), \(\Gen{E}[q](x,y) := q(x)\cup\{y\in\bigcup_{b\in f(x)}b \mid f\in E\}\cup \{y\nin\bigcup_{b\in g(x)}b \mid g\in\Balls[n-1]<[l]>(M)\text{ and }\forall f\in E,\,q(x)\vdash \bigcup_{b\in g(x)}b \subset \bigcup_{b\in f(x)}b\}\) is \(\LP(M)\)-definable and, by density of \(\K(M)\) in \(\K(\alg{M})\), \((a,c)\models\Gen{E}[p]\).
\end{proof}

We can now describe groups with definable generics in \(\SCVHG[p,e]\) in terms of groups definable in the algebraic closure.

\begin{theorem}[gp SCVH]
Let \(A \substr M\models\SCVHG[p,e]\) and let \(G\) be an \(\LGD(A)\)-interpretable group with a complete \(d\)-generic over \(A\). Then, there exists an \(\LG(A)\)-definable group \(H\) (defined in \(\alg{M}\)) and an \(\LGD(A)\)-definable group embedding \(f : G(M) \to H(M)\).
\end{theorem}

\begin{proof}
By Theorem\,\ref{EI SCVH}, the group \(G\) can be assumed to be \(\LGD(A)\)-definable. We now want to apply \ref{embed group}. Hypothesis\,\ref{hyp:tame dcl} follows from Theorem\,\ref{dcl SCVH} and, if we consider the pair \((\alg{M},M)\), Hypothesis\,\ref{hyp:descent} follows from the existence of the Frobenius automorphism on the field sort and the fact that, by density, both valued fields have the same purely geometric sorts.

There remains to prove Hypothesis\,\ref{hyp:ext def}. Pick an \(\LGD(M)\)-definable type \(p\in\TP<\LGD>(M)\). Let \(q_p\in\TP<\LG>(\alg{M})\) be given by \ref{ext def}. By \cite[Corollary\,1.7]{RidSim-NIP} and \cite[Proposition\,4.16]{HilKamRid}, we know that \(q_p\) is \(\LG(\alg{M})\)-definable. Since \(p\vdash q_p\), one can easily check that the construction of \(q_p\) is compatible with push-forwards by definable functions and the action of \(\aut[\LP](\alg{M})\).
\end{proof}

There are two natural improvements that could be expected of the previous theorem. First, can anything be said of groups definable in \(\SCVHG[p,e]\) that do not have definable generics. Although the general case seems quite out of reach, one could study weaker notions of genericity:

\begin{question}
Does \ref{gp SCVH} hold of groups with \(f\)-generics (equivalently, definably amenable groups)?
\end{question}

The second potential improvement comes from the non valued case. In \cite{BouDel-SCFGp}, Bouscaren and Delon show that a group definable in a separably closed valued fields of finite imperfection degree \(K\) is \emph{isomorphic} to the \(K\)-points of a group definable in \(\alg{K}\), i.e. an algebraic group. One can only wonder if the same is true in the valued setting:

\begin{question}
Can the embedding \(f\) in \ref{gp SCVH} be made surjective?
\end{question}

The method illustrated here could be applied to other settings, for example the model completion of valued differential fields (without any interaction). Some of the results regarding elimination of imaginaries in the geometric language, as well as the description of the definable closure are not known but they can be proven by adapting \cite{HilKamRid}.

%%% Local Variables:
%%% mode: latex
%%% TeX-master: "GpChk"
%%% End:

\printsymbols
\printbibli
\end{document}